\documentclass[11pt]{article}
\usepackage{amsfonts,amssymb,graphicx}
\usepackage{epsfig,amsmath,latexsym}
\usepackage{amscd, amssymb, amsthm, amsmath,eucal, verbatim}
\usepackage{epstopdf}
\usepackage{epstopdf}
\newtheorem{thm}{Theorem}[section]

\newtheorem{proposition}{Propostion}[section]

\begin{document}

\newcommand{\ct}{\cite}
\newcommand{\pr}{\protect\ref}
\newcommand{\su}{\subseteq}
\newcommand{\pa}{{\partial}}

\newcommand{\G}{{\mathbb G}}
\newcommand{\Z}{{\mathbb Z}}
\newcommand{\R}{{\mathbb R}}
\newcommand{\I}{{\mathrm{Id}}}
\newcommand{\lk}{{\mathcal{L}}}
\newcommand{\dg}{{\mathcal{D}}}

\newcounter{numb}

\title{Invariants of Knot Diagrams}
\author{Joel Hass\footnote {Supported in part by NSF grant DMS 3289292.}  ~and Tahl Nowik}

\date{\today}
 
\maketitle
\begin{abstract}
We construct a new order 1 invariant for knot diagrams.  
We use it to determine the minimal number of Reidemeister moves needed
to pass between certain pairs of knot diagrams. 
\end{abstract}

\section{Introduction}\label{A}
Oriented knots in $\R^3$ are usually represented by knot diagrams.  
Projecting a knot to a plane or 2-sphere in a generic direction
gives an immersed oriented planar or spherical curve with finitely many double points, or crossings.
A knot diagram is obtained by marking a neighborhood of each crossing to indicate which strand lies
above the other.  The higher strand is called the overcrossing and the lower one the undercrossing. Starting with a 
knot diagram, one can recover the original knot up to isotopy by constructing a curve with the overcrossing arcs pushed 
slightly above the plane of the diagram.

A central issue is to determine whether two knot diagrams represent the same knot, i.e.
whether the curves in $\R^3$ corresponding to each diagram are isotopic. If they represent the 
same knot we say that the two diagrams are {\em equivalent}.
Alexander and Briggs~\cite{AleBri26} and independently Reidemeister~\cite{R1} showed that equivalent 
diagrams can be connected through isotopy and  a series of three types of moves, usually referred to as  
Reidemeister moves. The number of such moves required to connect two equivalent diagrams, is difficult to  estimate. 
An exponential upper bound for the number of Reidemeister moves required to connect two equivalent diagrams 
is obtained in~\cite{HasLagPip}.  We can get some lower bounds by looking at crossing numbers, writhes and  winding numbers of diagrams since each Reidemeister move changes these numbers by 0, 1 or 2.  Less  obvious bounds are obtained in   \cite{Carter} and  \cite{Hayashi}.
 
In this paper we define a new family of knot diagram invariants, and focus on one of them in particular. 
We found these invariants by following the program of Arnold and 
Vassiliev for finite order invariants.  Our invariants are of order one.

As an application, for each $n$ we present two diagrams $D_n,E_n$ for the unknot, 
each with $2n+1$ crossings. For these two diagrams, which are almost identical, 
the writhe, cowrithe, crossing number and winding number
give a lower bound of 2 for the number of  Reidemeister moves required to pass from one to the other.  
Using our new invariant, 
we show that the minimal number of Reidemeister moves required to pass from $D_n$ to $E_n$ is  $2n+2$.
We also obtain restrictions on which Reidemeister moves may appear in any sequence of Reidemeister moves which realizes 
this minimum.

Our invariant takes values in a very large abelian group. It is natural to investigate $\Z$ valued invariants
obtained by composing it with homomorphisms into $\Z$. The ``cowrithe'' introduced in \cite{Hayashi} 
is obtained in this way. We obtain a relation of the cowrithe to Arnold's spherical curve invariants and the 
Alexander-Conway polynomial, which clarifies the limitations of the cowrithe for studying Reidemeister moves.

\section{The invariant}\label{B}

In what follows we consider two different types of geometric objects.
The first objects, which are the subject of study in this paper, are knot diagrams in $S^2$.
Two such diagrams are considered the same if they differ by an ambient isotopy of $S^2$. 
We denote the set of all such diagrams by $\dg$. Our goal is to construct invariants of 
knot diagrams. Towards that end we construct from a diagram a second geometric object, namely a two 
component link in $\R^3$.
This is a smooth embedding of $S^1 \coprod S^1$ in $\R^3$.
Two such embeddings are considered the same if they differ by an ambient isotopy 
of $\R^3$. 
We denote the set of all two component links by $\lk$, and the term \emph{links}
in this paper always refers to two component links.

Our basic construction relating knot diagrams to links is the following.
Given a knot diagram $D \in \dg$ and a crossing $a$ in $D$, define the smoothing $D^a \in \lk$,  to be the link obtained 
by smoothing the crossing $a$, i.e. performing a cut and paste on the four strands at the crossing that preserves the 
orientation of the arcs. The smoothing operation is independent of the orientation of the curve, 
since reversing orientation 
results in a change of orientation of both strands at the crossing. The diagram resulting from the smoothing is the 
diagram of an oriented 2-component link in $\R^3$, the link $D^a$. 
We order the components so that the first $S^1$ in $S^1 \coprod S^1$
is the component that enters a neighborhood of the crossing along the overcrossing arc and leaves it along  the undercrossing arc, see Figure~\ref{smoothing}. We define the sign of a crossing in a diagram in the usual way using the right-hand rule, so 
that each crossing in Figure~\ref{smoothing} is negative. 

\begin{figure}[htbp] 
   \centering
\includegraphics[width=3in] {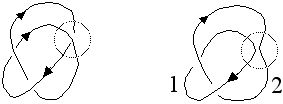}
  \caption{A smoothing results in a 2-component link}
 \label{smoothing}
\end{figure}

We now show how each invariant of 2-component links  gives rise to an invariant of knot diagrams.
Given a knot diagram $D$, denote by $D_+$ the set of all positive
crossings in $D$ and by   $D_-$ the set of all negative crossings.
Given an invariant $\phi:\lk \to S$ where $S$ is any set,
let $\G_S$ denote the free abelian group with basis $\{X_s, Y_s\}_{s \in S}$. We then define 
the invariant $I_\phi : \dg \to \G_S$ as follows:  
$$I_\phi(D) = \sum_{a \in D_+} X_{\phi(D^a)}  + \sum_{a \in D_-} Y_{\phi(D^a)}.$$

A particularly interesting example occurs when $\phi$ is taken to be $lk:\lk\to\Z$, the linking number of the two 
components of a link in $\lk$.
The resulting invariant is applied in the next section.

We now compute how $I_\phi(D)$ changes under Reidemeister moves on $D$.
Reidemeister moves RI, RII, RIII, are illustrated in Figure~\ref{reidmoves}.

\begin{figure}[htbp] 
   \centering
  \includegraphics[width=3in]{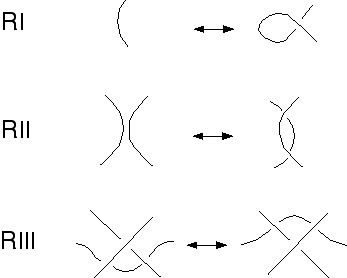} 
\caption{Reidemeister moves}
   \label{reidmoves}
\end{figure}

\noindent
RI: The contribution of all previously existing crossings is unchanged, and one new term is added.
The added term is $X_s$ if the new crossing created by the Reidemeister move is
positive and is $Y_s$ if the new crossing is negative, where $s$ is the value of $\phi$ on the link. 
See Figure~\ref{linkab}(a). 
In the case where $\phi=lk$, the added term is $X_0$ or $Y_0$, since the linking 
number of the smoothed link is $0$.

\noindent
RII: Again the contribution of all previously existing crossings is unchanged, 
but this time two new crossings are added.
There are two cases, depending on whether the orientations of the two strands participating in the 
Reidemeister move coincide or are opposite. 
We call these a \emph{matched} or \emph{unmatched} RII move, respectively.

For an unmatched RII move each of the two smoothings gives the same
link (see Figure~\ref{linkab}(b)), so the addition to the value of $I_\phi$ due to the Reidemeister move  
is of the form $X_s + Y_s$.
For $\phi=lk$ this gives  $X_n+Y_n$, where $n$ is the linking number of this
two-component link.

For a matched RII move, two different links appear from the two smoothings, differing from each other by one  
crossing change between the two components (see Figure~\ref{linkab}(c)). 
The addition to the value of $I_\phi$ is of the form $X_s+Y_t$ where 
$s$ and $t$ are the values of $\phi$ on these two links. When $\phi=lk$, the negative smoothing
produces a link with linking number greater by 1 than that produced by the positive smoothing, 
and so the added term is of the form $X_n+Y_{n+1}$.

\begin{figure}[htbp] 
   \centering
\includegraphics[width=4in]{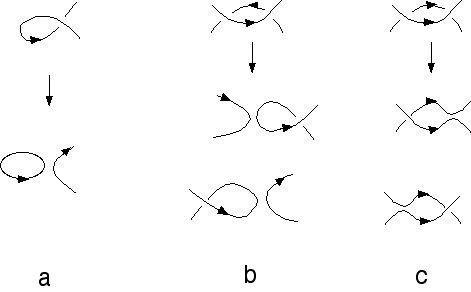} 
   \caption{Links arising from smoothing after RI and RII moves}
   \label{linkab}
\end{figure}

\noindent
RIII: To each crossing before the Reidemeister move there corresponds a crossing after the move.
For all crossings other than the three crossings participating in the move, the contribution to $I_\phi$
is clearly unchanged. 
As to these three crossings, the crossing between the top and middle strand also gives the same 
contribution before and after the move, since one can slide the bottom strand below the smoothed 
crossing, to show that the same link type is produced. An instance of this is demonstrated in 
Figure~\ref{link3} where diagrams $D,E$ are the diagrams before and after an RIII move. The crossing between the  
top and middle strands is marked $a$ and $D^a, E^a$ are seen to be 
the same link, via an isotopy corresponding to the original RIII move between $D$ and $E$.
By the same argument, the same is true for the crossing between the middle and bottom strands.

\begin{figure}[htbp] 
   \centering
\includegraphics[width=4in]{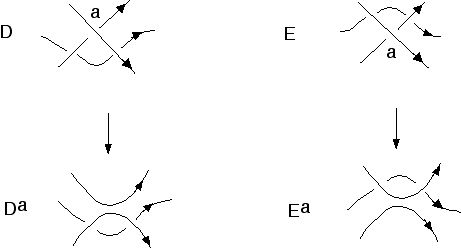}
   \caption{Links arising from smoothings of RIII moves involving the top and middle strand}
   \label{link3}.\end{figure}

The only contribution that  changes is that of the crossing between the 
top and bottom strands. If the crossing of the top and bottom strands is positive, then a term $X_s$ is
replaced by some term $X_t$, and so the change in the value of the invariant is of the form
$X_t - X_s$. In the same way,
if the crossing of the top and bottom strands is negative, the change is of the form $Y_t - Y_s$.
For the case $\phi =lk$,  one checks directly that the linking numbers of these two links
differ by precisely 1. Characteristic examples of the two basic cases appear in Figure~\ref{lktopbot}.  For the 
links $D^a, E^a$ in Figure~\ref{lktopbot}, three strands appear. The two that participated in the 
smoothing belong to different components, while the third strand (corresponding to the middle strand
in $D,E$) belongs to one or the other component. From this it is clear that $lk(D^a)$ and $lk(E^a)$ differ
by $\pm 1$. A similar analysis applies to the links $F^a, G^a$  in Figure~\ref{lktopbot}. 
It follows that the change in the value of $I_{lk}$ due to an RIII move 
is of the form $\pm(X_n - X_{n+1})$ or $\pm(Y_n - Y_{n+1})$. 
  
\begin{figure}[hbtp]
\centering
\begin{minipage}[c]{.49\textwidth}
\centering
\includegraphics[width=.95\textwidth]{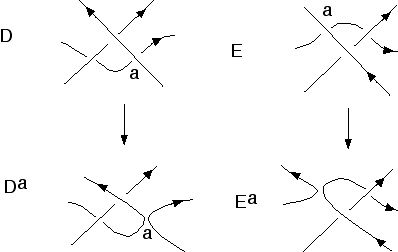}
\end{minipage}
\begin{minipage}[c]{.49\textwidth}
\centering
\includegraphics[width=.95\textwidth]{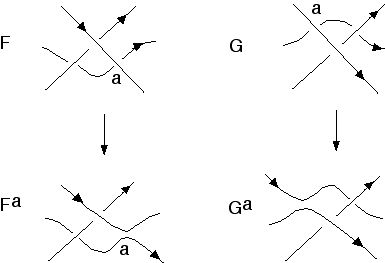}
\end{minipage}
\caption{The links arising from smoothings of RIII moves involving the top and bottom strand}
\label{lktopbot}
\end{figure}

We may use the above analysis to show that
for any $\phi$, the invariant $I_\phi$ is an order one invariant of knot diagrams.
Rather than giving a general definition of order $n$ invariants and then taking $n=1$,
we give the following equivalent definition. \\
{\bf Definition.} An invariant on knot diagrams has {\em order one} if 
whenever we may simultaneously perform two Reidemeister moves on a diagram,
in two disjoint discs $A,B \su S^2$ (that is, configurations as in Figure~\ref{reidmoves}  appear in $A$ and $B$), 
then the change in the invariant due to the move in $A$ is not 
affected by whether we first perform the move in $B$.

\begin{thm}\label{too}
For any link invariant $\phi:\lk \to S$, the invariant $I_\phi : \dg \to \G_S$ is an order one invariant.
\end{thm}

\begin{proof}
As seen in the analysis above, the change in the value of $I_{\phi}$ due to the Reidemeister move in $A$,
is determined by the link types obtained by smoothing either one or two crossings in $A$, 
before and/or after the move in $A$. The Reidemeister move performed in $B$, if done first,
does not affect the link type of these smoothed links, since a Reidemeister move in $B$ corresponds to an  isotopy on 
the smoothed links. 
\end{proof}

In the cases where $\phi$ is unchanged when permuting the two components of a link
and when reversing the orientation of a link, as is true for the invariant $lk$, then $I_\phi(D)$ is independent of 
the orientation of $D$. This is true since reversing the orientation of $D$ does not effect
the smoothing and the sign at each crossing. It only reverses the orientation of the 
smoothed links, and interchanges their 
two components. So, though the orientation of $D$ is used in the computation of $I_{lk}$, 
the invariant $I_{lk}$ is  an invariant of unoriented  diagrams.

The mirror image of a diagram $D$ is the diagram obtained from $D$ by reversing all of its crossings.
Taking the mirror image of a knot diagram $D$ has the following effect on $I_{lk}$. 
The smoothing at each crossing is the same, but the sign of all crossings is reversed. This interchanges
the $X$'s and $Y$'s, and reverses all linking numbers. So, the effect of taking mirror image is given
by mapping $X_n \mapsto Y_{-n}$ and  $Y_n \mapsto X_{-n}$. 
We also note that $I_{lk}$ is additive with respect to the operation of connected sum of diagrams.

\section{Application to knot diagrams}\label{C}

Let $R$ be the set of elements in $\G_\Z$ of the
form $X_0$, $Y_0$, $X_n+Y_n$, $X_n+Y_{n+1}$, $X_n - X_{n+1}$, $Y_n - Y_{n+1}$, and their negatives. 
By the analysis of Section \pr{B}, the elements of $R$ are precisely the elements of $\G_\Z$
that may appear as the change in the value of $I_{lk}(D)$ when performing a Reidemeister move on $D$.
Note that $R$ generates $\G_\Z$. The length of an element of $\G_\Z$ with respect to the generating 
set $R$ is called its {\em $R$-length}. Given two diagrams $D, E$ of the same knot, the $R$-length
of $I_{lk}(E) - I_{lk}(D)$ is a lower bound on the number of Reidemeister moves needed to get from $D$ to $E$. 
In particular, if $D$ is a diagram of the unknot, then the $R$-length of 
$I_{lk}(D)$ gives a lower bound on the number of Reidemeister
moves needed to get from $D$ to the trivial diagram.  
We point out that though we are working in the setting of spherical diagrams, 
the same lower bounds apply, a fortiori, to planar diagrams.

We  now use the above procedure to determine the minimal number of Reidemeister moves needed to pass between  certain pairs of diagrams of the unknot.

\begin{figure}[htbp] 
   \centering
\includegraphics[width=5in]{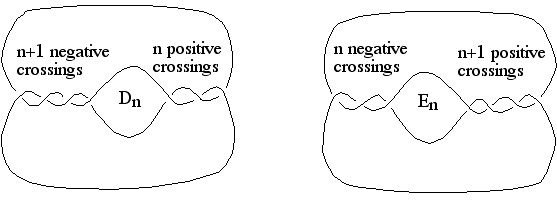} 
   \caption{An unknot $D_n$ and its mirror image $E_n$}
   \label{newDnEn}
\end{figure}

Let $D_n$ be the knot diagram appearing in Figure~\ref{newDnEn} (for $n=3$). It has $2n+1$ crossings, where the first
$n+1$ crossings, from left to right, are negative, and the following $n$ crossings are positive.
Let $E_n$ be the knot diagram obtained from $D_n$ by reversing the middle crossing, so that $E_n$ has from left to right, 
$n$ negative crossings followed by $n+1$ positive crossings.
Note that $E_n$ is a mirror image of $D_n$. 
The writhes of $D_n$ and $E_n$ differ by two. Their winding numbers, 
crossing numbers and cowrithes are the same 
(for definition of cowrithe see Section \pr{D}). So these invariants can only tell us that
at least two Reidemeister moves are needed to pass between these diagrams. 
It is easy to see how to arrive from $D_n$ to $E_n$ with $2n+2$ Reidemeister moves, namely, 
perform $n$ RII moves and one RI move to 
arrive from $D_n$ to the trivial diagram, and then one RI move and $n$ RII moves to create $E_n$.
We now prove that there is no shorter way:

\begin{thm}\label{DnEn}
Let $D_n, E_n$ be the two knot diagrams with $2n+1$ crossings appearing in Figure~\ref{newDnEn}. 
Then the minimal number of Reidemeister moves required to arrive from $D_n$ to $E_n$ is
$2n+2$. 
Furthermore, any sequence of Reidemeister moves realizing this minimum involves
precisely two RI moves.  The other $2n$ moves are of type corresponding 
to the following four elements of $R$:
$X_0+Y_1$, $-(X_{-1}+Y_0)$, $X_0 - X_{-1}$, $Y_1-Y_0$. 
In particular, no unmatched RII move may appear. 
\end{thm}

\begin{proof}
By direct computation we see that
$$
I_{lk}(D_n) =(n+1)Y_0 + nX_{-1},
$$
and 
$$
I_{lk}(E_n)= (n+1)X_0 + nY_1,
$$
so
$$I_{lk}(E_n) - I_{lk}(D_n) = (n+1)X_0 + nY_1 - (n+1)Y_0 - nX_{-1}.$$ 
We denote this element by $v$.
We first show that the $R$-length of $v$ is $2n+2$. 
Let $g:\G_\Z \to\Z$ be the homomorphism defined by setting 
$g(X_0)=1$, $g(Y_0)=-1$, and $g(X_m)=g(Y_m)=0$ for all $m \neq 0$. 
Then $g(v)=2n+2$, and $|g(r)| \leq 1$ for any $r \in R$. 
It follows that the $R$-length of $v$ is at least $2n+2$.
The sequence of Reidemeister moves described above 
realizes this minimum, namely $v=-n(X_{-1}+Y_0)-Y_0+X_0+n(X_0+Y_1)$.

We now show that any presentation
of $v$ as a sum of $2n+2$ elements of $R$ involves precisely two terms of the form $X_0$ and/or $-Y_0$,
and all other $2n$ terms are of the four types appearing in the statement of the theorem.

Let $f:\G_\Z \to \Z$ be the homomorphism defined by $f(X_m)=1,f(Y_m)=-1$ for all $m$.
Then $f(X_0)=1,f(Y_0)=-1$ and the value of $f$ on all other elements of $R$ is 0.
Since $f(v)=2$ we see we need at least two terms of the form $X_0$ or $-Y_0$ to present $v$.
(This is simply a writhe argument, since $f \circ I_{lk}$ is the writhe.)

Now, let $e:\G_\Z \to\Z$ be the homomorphism defined by setting 
$e(X_0)=e(Y_1)=1$, $e(X_{-1})=e(Y_0)=-1$, and $e(X_m)=e(Y_m)=0$ for all other $X_m,Y_m$. 
Then $e(v)=4n+2$. We have $e(X_0)=1,e(Y_0)=-1$, and for all elements
of $R$, $|e(r)| \leq 2$, where the only elements in $R$ for which the value of $e$ is precisely 2
are $X_0+Y_1$, $-(X_{-1}+Y_0)$, $X_0 - X_{-1}$, $Y_1-Y_0$. 
We already know that at least two terms of the form $X_0, -Y_0$ appear in any presentation of $v$, 
each contributing only 1 to $e$, 
and so the other $2n$ terms of a minimal presentation
must each contribute $2$ in order to get to $4n+2$. 
So, the other $2n$ terms  must be of
the form $X_0+Y_1$, $-(X_{-1}+Y_0)$, $X_0 - X_{-1}$, $Y_1-Y_0$. 

As to the concluding remark, since none of these four elements are of the form $X_n + Y_n$,  
no unmatched RII move may appear.
\end{proof}

\section{Relation to curve and knot invariants}\label{D}

In the proof of Theorem \pr{DnEn} we have used three different homomorphisms $g,f,e:\G_\Z \to \Z$ in order to analyze the 
element $v$. $g$ and $e$ were ad hoc choices constructed especially for this specific $v$. But as noted,
$f$ is of general interest, $f \circ I_{lk}$ being the writhe of the diagram. Another such substitution 
is $k(X_n)=1,k(Y_n)=1$ for all $n$. The resulting diagram invariant $k \circ I_{lk}$ is then the crossing number of the diagram.
In this section we are interested in the substitution $h:\G_\Z \to \Z$ given by $h(X_n)=-n, h(Y_n)=n$ for all $n$.
For a knot diagram $D$ we denote $H(D) = h \circ I_{lk}(D)$. This invariant of knot diagrams 
appears in \cite{PolyakViro01} and in \cite{Hayashi} (with opposite sign) where it is named the \emph{cowrithe}. 
We now study the properties of $H$, and understand its limitations in establishing lower bounds 
for the number of Reidemeister moves between knot diagrams. We will see that $H$ may be presented
as $H=G_1 + G_2$ where $G_1$ depends only on the knot type of $D$, and $G_2$ depends only on the  underlying spherical 
curve of $D$. 
Now, if we are interested in the number of Reidemeister moves between two diagrams $D,E$, 
then all diagrams considered are of the same knot type, so for the sake of this analysis $G_1$ is a constant and we are 
left with $G_2$ which is only an invariant of the underlying spherical curves. 
So any lower bound obtained from $H$ to the number of Reidemeister moves required to pass between $D$ and $E$ 
is in fact a lower bound to the number of moves on curves required to pass between the 
underlying curves of $D$ and $E$.   
For example, for the pair $D_n, E_n$ of Section \pr{C}, $H(D_n)=H(E_n)$ since they are each diagrams of the same knot 
type and have the same underlying curve, so $H$ gives 0 as lower bound.

The invariants $G_1, G_2$ referred to, are well known invariants. $G_1$ is $-4$ times the coefficient of $x^2$ in the 
Conway polynomial, and $G_2$ is Arnold's invariant of spherical curves $St+  J^+/2$, normalized  to have value 0 on the 
two curves appearing in Figure~\ref{normalize}, the circle and the figure eight  curve. (These are representatives of the
two regular homotopy classes of immersions of $S^1$ into $S^2$.) We denote by $A$ the invariant $St+   J^+/2$ with this 
normalization.
We mention that Arnold originally considered curves in the plane in \cite{Arnold94}, 
but then observed that his invariants may  
be defined for curves in $S^2$. Indeed they are well defined on $S^2$, since they are locally well defined, as in the 
plane,
and since the components of the space of ``normed'' immersions of $S^1$ into $S^2$  defined in \cite{Arnold94} are 
simply connected.

Let $\widehat{D}$ denote the underlying spherical curve of a knot diagram $D$, 
and let $c_2(D)$ denote the coefficient of $x^2$ in the 
Conway polynomial of the knot represented by $D$. 
Then the relation mentioned is 
$$H(D) = -4c_2(D) + A(\widehat{D}).$$ 
This relation follows from \cite{Pol}, as we will show. (Various similar observations appear in \cite{PolyakViro01}.) 
We then give a self contained proof of this relation, which demonstrates the special properties of 
$H$ and $A$ and their relation to the properties of $I_{lk}$.

We first show that $H$ indeed coincides (up to sign) with the cowrithe defined in \cite{Hayashi}.
To a knot diagram $D$ there is attached a chord diagram, where we connect two points of the 
domain $S^1 \subset \R^2$ by a straight chord in $\R^2$, if they are mapped into the same point in the diagram $D$. 
For two crossings $a,b$ 
in $D$, $b$ is a crossing between the two distinct components of $D^a$ precisely 
when the chords corresponding to $a$ and $b$ cross in the chord diagram.
We denote this (symmetric) relation between $a$ and $b$ by $x(a,b)$.
So $b$ contributes $\pm 1/2$ to $lk(D^a)$ according to the sign of $b$ which we denote $sgn(b)$,
precisely when $x(a,b)$.
It follows that 
$$lk(D^a) = \frac{1}{2}  \sum_{\{b:x(a,b)\}} sgn(b).$$
From this it follows that $$H(D) = \sum_{ \{ \{a,b\} :x(a,b) \}} -sgn(a)sgn(b)$$
(Note that $\{a,b\}$ appearing under the summation sign is a set, not an ordered pair.)
We see that indeed $H$ is minus the cowrithe.

From these formulas for $lk(D^a)$ and $H(D)$ we can see how $H$ changes under a crossing change.
Let $D_{a+},D_{a-}$ be two diagrams
differing only by a crossing change at $a$, the crossing at $a$ for $D_{a+}, D_{a-}$ being positive 
and negative, respectively, and let $D^a$ be the common smoothing of $D_{a\pm}$ at $a$.
Then the following is clear from the above two formulas: 

\begin{proposition}\label{skein}
$H(D_{a+}) - H(D_{a-}) = -4lk(D^a)$
\end{proposition}

If an invariant $i$ of diagrams satisfies the skein relation appearing in Proposition \pr{skein}, and in addition
$i(D)=0$ for any diagram of the unknot, then by \ct{k} Chapter III it must coincide with $-4c_2(D)$.
In fact it is enough that $i(D)=0$ for any descending knot diagram.
Define 
$$i(D) = H(D) - A(\widehat{D}).$$
Our goal is to show that $i(D) = -4c_2(D)$.
Since $\widehat{D_{a+}} = \widehat{D_{a-}}$ then by Proposition \pr{skein} we have $i(D_{a+}) - i(D_{a-}) = -4lk(D^a)$. 
It remains to show that $i(D)=0$ for descending diagrams.
This follows directly from \cite{Pol} Corollary 2, if one chooses the basepoint required there
at the initial point of descent. 

We now present a self contained proof that $i$ is a knot invariant, 
by showing that it is invariant under all Reidemeister moves.
The proof is independent 
of \cite{Pol} and of the fact that $c_2$ is a knot invariant. 
Knowing that $i(D)$ is a knot invariant will also reprove that $i(D)=-4c_2(D)$ 
using the above argument, since $i(D)=0$ for the trivial diagram, 
and therefore for any diagram of the unknot.

We first note the values of $h$ on $R$:
$h(X_0)=0$, $h(Y_0)=0$,
$h(X_n+Y_n)=0$, $h(X_n + Y_{n+1}) = 1$, $h(X_n - X_{n+1}) = 1$ and $h(Y_n - Y_{n+1}) = -1$. 
We thus see that $H$ remains unchanged when a diagram changes by an  
RI  move and an unmatched RII move, and
increases by 1 when a diagram changes  by a matched RII move. 
Furthermore, $H$ changes by $\pm 1$ when the  diagram changes  by an RIII move, and 
inspection shows that the change of $1$ or $-1$ precisely coincides with the 
change of $1$ or $-1$ of Arnold's ``strangeness'' invariant $St$ for the underlying curve.
Two cases are indicated in Figure~\ref{lktopbot}. For each case one needs to check 
both possible cyclic orderings of the three strands along $S^1$.
All other cases are obtained from these two by reversing the orientation of the middle strand, and by reversing the 
crossing between the top and bottom strands while adjusting the middle strand accordingly.
So, using the same names for the moves on curves as for the corresponding moves on knot diagrams,
we have shown that the change in $H(D)$ and $A(\widehat{D})$ is the same under moves RII, RIII. 

As to RI moves, the change in $H$ is 0, and we now check the change for $A$. 
We point out that an RI move on spherical curves changes the regular homotopy class of the curve, and so 
asking about the change under an RI move is only meaningful when a specific normalization is chosen for the two regular 
homotopy classes.

\begin{figure}[htbp] 
   \centering
    \includegraphics[width=2in]{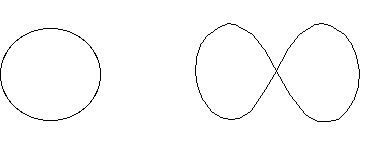} 
   \caption{The invariant $A$ is normalized to have value 0 on these two curves}
   \label{normalize}
\end{figure}

We refer to a small loop in a curve as a {\em kink}.

\begin{proposition}\label{p1}
The invariant $A$ is invariant under a move that slides an arc across a kink, as in Figure~\ref{2moves}.
\end{proposition}

\begin{figure}[htbp] 
   \centering
    \includegraphics[width=2in]{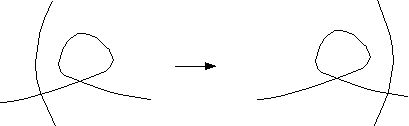} 
   \caption{A   kink slides  across an arc.}
   \label{2moves}
\end{figure}

\begin {proof}
See Figure~\ref{2invariant} for the  two cases to be checked.
\end{proof}

 \begin{figure}[htbp] 
   \centering
\includegraphics[width=4in]{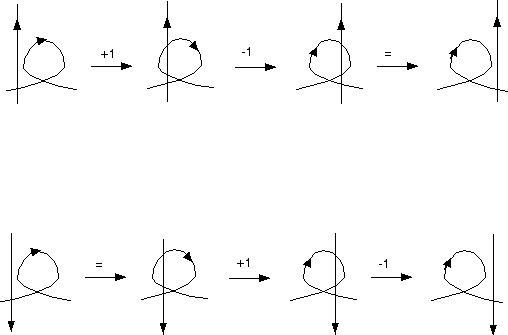} 
   \caption{The value of $A$  is invariant under these moves.}
   \label{2invariant}
\end{figure}

\begin{proposition}\label{p2}
The invariant $A$ is invariant under an RI move, that is, the introduction of a kink,  
the move shown in Figure~\ref{kink}. 
\end{proposition}

\begin{figure}[htbp] 
   \centering
\includegraphics[width=3in]{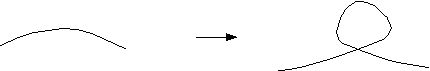} 
   \caption{A kink is created.}
   \label{kink}
\end{figure}

\begin{proof}
One first checks directly that the introduction of a kink to the embedded circle or to the figure 
eight curve leaves $A$ unchanged, as shown in Figure~\ref{addkink}.

\begin{figure}[htbp] 
  \centering
\includegraphics[width=3in]{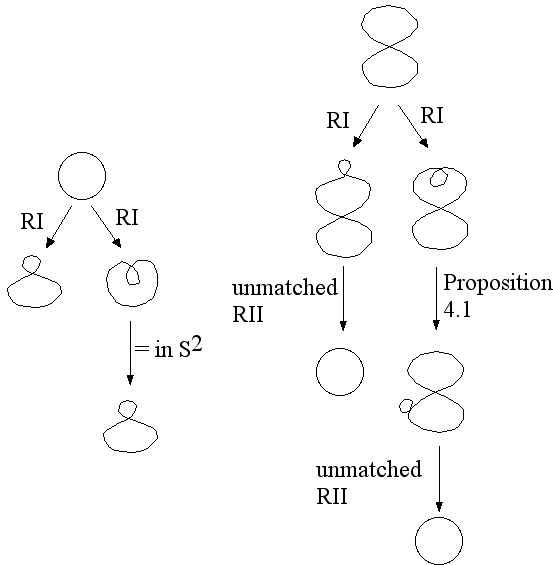} 
   \caption{$A$  is unchanged by the introduction of a kink.}
   \label{addkink}
\end{figure}

Now given a smooth general position curve $C \subset S^2$, there is a regular homotopy $F$ in $S^2$ taking $C$ either 
to the circle or to the figure eight. 
Such regular homotopy involves only RII and RIII moves, and
the sum of contributions of all these moves along $F$ is 
precisely $-A(C)$.
Let $C'$ be a curve obtained from $C$ by an RI move, and apply the same regular homotopy $F$ to $C'$, carrying 
along the added kink. We experience the same moves along $F$ except that occasionally a strand needs to 
pass the additional kink as in Figure~\ref{2moves}. By 
Proposition~\ref{p1} this occurrence does not change the value of $A$, and so the total contribution of all 
moves is again $-A(C)$. We arrive at a curve which is obtained from the embedded circle or figure 
eight by an RI move,
and we have verified that the value of $A$ on such a curve is 0.
We conclude that $A(C')=A(C)$.
\end{proof}

Proposition \pr{p2} establishes that the
change due to an RI move is the same for $H(D)$ and $A(\widehat{D})$. Together with the above similar observation 
regarding RII and RIII moves, we see that $i(D) = H(D) - A(\widehat{D})$ is invariant under all Reidemeister moves,
proving that $i(D)$ is a knot invariant.

\vspace*{.5in}
\begin{tabular}{ll}
{\tt email}: & {\tt hass@math.ucdavis.edu} \\ & {\tt tahl@math.biu.ac.il } \\
~~~ \\
~~~ \\
{\tt address}: & J. Hass \\
& Dept. of Mathematics \\
& Univ. of California, Davis \\
& Davis, CA 95616 \\
~~~ \\
& T. Nowik \\
& Dept. of Mathematics \\
& Bar-Ilan University \\
& Ramat-Gan 52900, Israel \\
 \\
\end{tabular}


\begin{thebibliography}{StoPC}

\bibitem{Adams}
C. Adams,
{\em The Knot Book. An elementary introduction to the mathematical theory
of knots},
W. H. Freeman, New York, 1994.

\bibitem{AleBri26}
J. W. Alexander and G. B. Briggs,
{\em On types of knotted curves},
Ann. Math., {\bf 28} (1926/27), 562--586.

\bibitem{Arnold94}
V. I. Arnold: 
{\em Plane Curves, Their Invariants, Perestroikas and Classifications,}
 Advances in Soviet Mathematics, {\bf  21} (1994), 33-91.

\bibitem{Carter}
J. S. Carter, M. Elhamdadi ,  M. Saito and S. Satoh,
{\em A lower bound for the number of Reidemeister moves of type III}
to appear in Topology and Its Applications.

\bibitem{ChmutovDuzhin}
S. Chmutov, S. Duzhin. 
{\em Explicit formulas for Arnold's generic curve invariants,}
In: ``Arnold-Gelfand Mathematical Seminars: Geometry and Singularity Theory." Birkhauser, 1997, 123--138.

\bibitem{Goe34}
L. Goeritz,
{\em Bemerkungen zur knotentheorie,}
Abh. Math. Sem. Univ. Hamburg, {\bf 10} (1934), 201--210.

\bibitem{GoussarovPolyakViro:2000} 
M. Goussarov, M. Polyak, O. Viro, 
{\em Finite-type invariants of classical and
virtual knots,}
Topology 39 (2000) 1045--1068.

\bibitem{Hagge} 
T.J. Hagge, 
{\em Every Reidemeister move is needed for each knot type,}
arxives:math.GT/0404145.

\bibitem{Hak61}
W. Haken,
{\em Theorie der Normalflachen, ein Isotopie Kriterium f\"ur ein Kreis,}
Acta Math.,  {\bf 105} (1961), 245--375.

\bibitem{Hass97}
J. Hass,
{\em Algorithms for recognizing knots and 3-manifolds,}
to appear in Chaos, Fractals and Solitons.

\bibitem{HasLagPip}
J. Hass and J. C. Lagarias,
{\em The number of Reidemeister moves needed for unknotting,}
J. Amer. Math. Soc. 14 (2001), no. 2, 399--428.

\bibitem{Hayashi}
C. Hayashi,  
{\em A lower bound for the number of Reidemeister moves for unknotting.} 
J. Knot Theory Ramifications 15 (2006), no. 3, 313--325.


\bibitem{k}
L. H. Kauffman,
{\em On Knots},
Annals of Mathematics Studies 115, Princeton University Press, 1987.

\bibitem{Ostlund}
O. \"{O}stlund,
{\em Invariants of knot diagrams and relations among Reidemeister
moves,} 
J. Knot Theory Ramifications, 10 (2001), no. 8, 1215--1227. 

\bibitem{Pol}
M. Polyak,
{\em Invariants of curves and fronts via Gauss diagrams.}
Topology 37 (1998), no. 5, 989--1009. 


\bibitem{PolyakViro94}
M. Polyak and O. Viro, 
{\em Gauss diagram formulas for Vassiliev invariants,}  Internat.
Math. Res. Notices (1994) No. 11  445--453.  

\bibitem{PolyakViro01}
M. Polyak and O. Viro,
{\em On the Casson knot invariant.} Knots in Hellas '98, Vol. 3 (Delphi). 
J. Knot Theory Ramifications 10 (2001), no. 5, 711--738.


\bibitem{R1}
H. Reidemeister, 
{\em Knoten und Gruppen}, Abh. Math. Sem., Univ. Hamburg,
{\bf 5} (1926), 7--23.

\bibitem{Trace}
Bruce Trace, 
{\em On the Reidemeister moves of a classical knot,}
Proc. Amer. Math. Soc. {\bf 89} (1983), no. 4, 722--724.

\end{thebibliography}
\end{document}